\documentclass[a4paper,fleqn,oneside]{article}

\usepackage{mathtools,amssymb,amsfonts,amsthm,mathrsfs}
\usepackage{latexsym,graphicx}
\usepackage{booktabs,lscape}
\usepackage[T1]{fontenc}

\newcommand{\be}{\begin{eqnarray}}
\newcommand{\ee}{\end{eqnarray}}
\newcommand{\bse}{\begin{subequations}}
\newcommand{\ese}{\end{subequations}}
\newcommand{\bdm}{\begin{displaymath}}
\newcommand{\edm}{\end{displaymath}}
\numberwithin{equation}{subsection}
\theoremstyle{plain}

\theoremstyle{plain}
\newtheorem{pro}{Proposition}[subsection]
\theoremstyle{definition}
\newtheorem{defn}{Definition}[subsection]
\theoremstyle{remark}
\newtheorem*{ack}{Acknowledgement}
\newtheorem*{note}{Note added}
\newtheorem*{conv}{Conventions adopted}

\newcommand{\Ksi}{\boldsymbol{\xi}}
\newcommand{\Ze}{\boldsymbol{\zeta}}
\newcommand{\X}{\mbox{\boldmath $X$}}
\newcommand{\Si}{\boldsymbol{\sigma}}
\newcommand{\e}{\mbox{\boldmath $e$}}
\newcommand{\Th}{\boldsymbol{\theta}}
\newcommand{\Al}{\boldsymbol{\alpha}}
\newcommand{\Be}{\boldsymbol{\beta}}
\newcommand{\Ga}{\boldsymbol{\Gamma}}
\newcommand{\Om}{\boldsymbol{\Omega}}

\begin{document}
\title{\Large\textbf{Locally Homogeneous Spaces, Induced Killing Vector Fields and Applications to Bianchi Prototypes}}
\vspace{1cm}
\author{\textbf{G.O. Papadopoulos}\thanks{e-mail: gopapado@phys.uoa.gr}\\
\textit{National \& Kapodistrian University of Athens, Physics Department}\\
\textit{Nuclear \& Particle Physics Section}\\
\textit{Panepistimioupolis, Ilisia GR 157--71, Athens, Greece}
\vspace{0.5cm}\\
\textbf{Th. Grammenos}\thanks{e-mail: thgramme@civ.uth.gr}\\
\textit{University of Thessaly, Department of Civil Engineering}\\
\textit{GR 383--34, Volos, Greece}\\}
\date{}
\maketitle
\begin{center}
\textit{\textbf{Dedication}}
\textit{G.O. Papadopoulos wishes to dedicate this work to Gerasimos}.
\end{center}

\begin{abstract}

An answer to the question: \emph{Can, in general, the adoption of a given symmetry induce a further symmetry, which might be hidden at a first level?} has been attempted in the context of differential geometry of locally homogeneous spaces. Based on \'{E}. Cartan's theory of moving frames, a methodology for finding all symmetries for any $n$-dimensional locally homogeneous space is provided. The analysis is applied to 3-dimensional spaces, whereby the embedding of them into a 4-dimensional Lorentzian manifold is examined and special solutions to Einstein's field equations are recovered.\\
The analysis is mainly of local character, since the interest is focused on local structures based on differential equations (and their symmetries), rather than on the implications of, e.g., the analytic continuation of their solution(s) and their dynamics in the large.

\vspace{0.3cm}
\noindent
\textbf{MSC-Class} (2010): 83C05, 83C15, 83C20, 53B20, 53C30, 58J70, 22E65, 22E70
\end{abstract}


\section{Introduction}

A long time ago, a not well known discovery was made (for details, see: \cite{Daskalos}): in the Bianchi Type III cosmological prototype, a fourth Killing vector field of the 3-dimensional locally homogeneous space has been discovered and the adoption of that, through proper prolongation, has led to a special  (of the Kinnersley type, see \cite{Kinnersley}) solution to the Einstein's field equations (EFEs). Prompted by this, a question has arisen: \emph{Can, in general, the adoption of a given symmetry induce a further symmetry, which might be hidden at a first level?}\\
At least to the best of the authors' knowledge, the only work discussing the homogeneous case in 3 dimensions is \cite{Szafron}; but the treatment is not systematic (i.e., no general method, applicable to $n$ dimensions, is presented), and from a very different point of view\footnote{It is deemed appropriate to quote the abstract of reference \cite{Szafron}: \textit{We derive necessary and sufficient tensor conditions for the existence of a four parameter isometry group $G_{4}$ which acts multiply transitively on a Riemannian $V_{3}$. We then apply these results to determine which spatially homogeneous cosmological models have induced 3-metrics which are invariant under such a four parameter group}.}. Also, another similar work is to be found in \cite{Milnor}, but the topological issues put severe limitations on the results.

The idea of symmetry, as realised through the context and the applications of group theory, has been proven to be extremely fruitful in many aspects and in many fields of research. From the realm of differential equations to the areas of geometry and topology, and from simple (or even complex) technical problems to modern physical theories, symmetry has been a most significant component in the effort towards understanding the nature of versatile problems and, consequently, providing a solving process --to some extent, at least.

Of course, although the interest in the present parer is focused on both the Riemannian geometry of locally homogeneous spaces (LHS), from a pragmatic point of view, and the EFEs within the context of the former, the following short thoughts seem to be quite general.\\
In a very broad sense, one could say that there are two cases where the idea of symmetry is implemented:
\begin{itemize}
\item[$C_{1}$] In the first case, symmetry can be considered as an exact assumption which can lead to a simplified, compared to the initial, problem with the hope of an equally simple  
                          and exact, as a statement, solution. A typical example for this case is the adoption of some symmetry (as an accurate assumption) upon solving the EFEs. Indeed, the existence 
                          of some Killing vector field(s) (KVF(s)) leads (in many instances) to analytic, closed in form, solutions to the EFEs --see \cite{Stephani_et_al.} for a panoramic view on this 
                          issue.
\item[$C_{2}$] In the second case, symmetry is a (usually ``\mbox{}hidden'') feature characterising a given system the implementation of which can clarify as well as simplify the 
                          problem to be solved. Here, the archetype\footnote{It is well known that Sophus Lie gave to the notion of continuous symmetry a rigorous meaning by inventing the theory of continuous groups 
                          of transformations and applying his ideas to the theory of differential equations in order to generalise Evarist Galois' theory on algebraic equations. Hence, the word archetype is the most 
                          appropriate.} is the discovery (using various standard methods) of, e.g., some Lie-point symmetries admitted by a system of differential 
                          equations; their existence signals a reduction (again, using standard methods) of the initial system to a simpler final system. At a conceptual level, one could say that the 
                          initial differential equations are given modulo some kind of redundancy, and that symmetry ---through the reduction process---  results in an irreducible, yet 
                          equivalent, system of differential equations, unveiling ---at the same time--- the true degrees of freedom.
\end{itemize}
A synthesis of these two cases (perhaps, best imagined through the merging of their corresponding examples --which are not chosen arbitrarily) can serve as the thread leading to a third case, worth to be studied per se. Indeed, a reasonable question will be \emph{whether the adoption of a given symmetry, \underline{and no other}, might lead to the discovery of another, extra symmetry}.  

Returning to the thematics of interest, which is the Riemannian geometry of the LHS in $n$ dimensions, the aforementioned considerations can be restated as follows\footnote{The following discussion is not rigorous; it rather offers a flavour, on general grounds, of what is going to be the central scope of this paper. Therefore, a naive approach in both the formalism and definitions/terminology  is adopted.}:\\
Generally, a (pseudo)Riemannian geometry $(\mathcal{M}, \mathbf{g})$, where $\mathcal{M}$ is a smooth $n$-dimensional manifold, and $\mathbf{g}$ a smooth metric tensor field, defined throughout $\mathcal{M}$, defines a LHS in $n$ dimensions if the Lie algebra of KVFs spans the tangent space at all points of  $\mathcal{M}$. Then, locally, the space is completely described by the structure constants of the underlying Lie algebra. These constants satisfy algebraic constraints coming from the Jacobi identity.
Of course, the KVFs satisfy the conditions:
\be\label{KillingEquations}
\pounds_{\Ksi_{A}}\mathbf{g}=0
\ee

Now, it is easy to adapt the previous general comments on symmetries. 
In principle, there are two ways to consider a symmetry expression like $\pounds_{\Ze}\mathbf{g}=0$ :
\begin{itemize}
\item[$HC_{1}$] Either as a system of partial differential equations (PDEs) of the 1st order which, given some KVF(s) $\{\Ze\}$ as an initial symmetry assumption, is to be solved in  
                            terms of the metric tensor field components.
\item[$HC_{2}$] Or as a system of PDEs of the 1st order which, given a smooth metric tensor field, is to be solved in terms of the unknown(s) $\{\Ze\}$.
\end{itemize}
One could say that for the case of LHS, and since all the KVFs are given each time (by virtue of the very definition for a locally homogeneous Riemannian geometry), the first possibility describes exactly the state of affairs, while the second is empty (or meaningless). But, on the other hand, the question posed earlier still remains:\\
\emph{Let an $n$-fold $\{\Ksi_{A}\}$ of KVFs, corresponding to a given Lie algebra $\mathfrak{g}_{n}$, be \underline{the only} initial symmetry assumption for a locally homogeneous Riemannian geometry in $n$ dimensions. Let also a metric tensor field $\mathbf{g}_{\text{red.}}$ which solves \eqref{KillingEquations}. Under such a setting, are there other, non trivial solutions to the system $\pounds_{\Ze}\mathbf{g}_{\text{red.}}=0$, in terms of some KVF(s) $\{\Ze\}$ ?}\\
The paradigm of Bianchi Type III (see, e.g., \cite{Daskalos}), surprisingly enough, rendered this question not only non trivial ---because the answer there is affirmative--- but also justified --since the discovery of the extra KVF has led to a special solution to EFEs for that prototype \cite{Daskalos,Kinnersley}).

Another perspective leading to this question is provided by the following general comments:\\
Suppose that one wants to solve a symmetry condition of the form $\pounds_{\Ze}\mathbf{g}=0$ by implementing formal solving processes for systems of PDEs. Then, in principle, two equally formal and associated problems will appear.\\
The first problem, which is twofold in nature, is that (after a formal solving process) redundant degrees of freedom might emerge --as the paradigm of the Schwarzschild space time shows. Indeed, substitution of all the three KVFs (defining the Lie algebra of the group $SO(3,\mathbb{R})$ acting multiply transitively on 2-dimensional subspaces) into the symmetry condition, results in a metric tensor field with one spurious component; then, a general coordinate transformation (GCT) must be implemented in order to absorb this component. The twofold nature of this problem is to be found in the fact that, modulo a few instances, no one knows in advance if and how many spurious degrees might appear; moreover it is not always easy to find the needed GCT(s).\\
The second problem is that a potential redundancy renders the search for further induced symmetries much harder --and such further symmetries might exist, as the example of the Bianchi Type III homogeneous space shows.

However, this is not the end of the story. The consequences of this simple question, posed for any LHS in $n$ dimensions (which is the objective of Section 3), have important implications:
\begin{itemize}
\item[$I_{1}$] One realises that it is necessary to know all the initial symmetries in a given setting, especially when the existence of some of them is a mere consequence of the adoption of the rest.
\item[$I_{2}$] Knowledge of the full symmetry group is tantamount to the knowledge of not only the full gauge freedom but also of the true degrees of freedom --see \cite{Gauge}.
\item[$I_{3}$] Points $I_{1}$ and $I_{2}$ remain valid even when global/topological considerations enter the analysis --see \cite{Global,GlobalandGauge}
\end{itemize}

Section 2 provides some, necessary for the development in Sections 3 and 4, mathematical preliminaries; essentially, elements from \'{E}. Cartan's theory of moving frames are presented.\\
The objective of Section 3 is to answer in detail the previous question (in its last form) for any LHS in $n$ dimensions and to discuss in full the implications of the extra symmetry (when present) along the lines described above.\\
In Section 4 two simple applications are given: (a) the methodology is applied in the special case when $n=3$ (i.e., Bianchi prototypes), and (b) the embedding of the 3-dimensional LHS into a 4-dimensional Lorentzian manifold (Bianchi cosmological models) is considered. Then, the extra symmetry (when present, and by means of proper prolongations) is used to lead to special, vacuum (i.e., Ricci flat) solutions to EFEs, which might be of interest. \\
It should be stressed that the character of the present study is local because the desired result is to communicate the basic ideas in the simplest possible form, and thus to avoid more technical, subtle matters related to topological issues. On the other hand, many results are independent of the topology. Of course, again, the methodology is susceptible of proper modifications in order to remain valid when global considerations are to be taken into account, and this might be the goal of a future work. 


\section{\'{E}. Cartan's theory of moving frames: Elements}

For the sake of both completeness and logical continuity, a short collection of elements (i.e., basic notions, definitions, and results) from \'{E}. Cartan's theory of moving frames will be given. It is meant, of course, that the purpose of this short presentation is not to give a detailed account on the theory; it rather serves as a reference collection. For extended and versatile treatments on the subject, see e.g., the last two from references \cite{ContinuousGroupsOfTransoformations}. The original reference is, of course, \cite{Cartan}.

\begin{conv}
Lower case Latin indices are used for any coordinate space in $n$ dimensions, while capital Latin indices for either any (co)tangent space in $n$ dimensions or any $n$ parametric group space. Both classes of indices have as their domain of definition the set $\{1,2,\ldots n\}$. 
\end{conv} 

Let $\mathcal{R}$ be a (semi)Riemannian space described by the pair $(\mathcal{M},\mathbf{g})$, where $\mathcal{M}$ is an $n$ dimensional, simply connected\footnote{The adoption of this assumption is prompted by a potential implementation of \emph{Poincar\'{e}'s Lemma}. Alternatively, this constraint can be replaced by another one, less restrictive, by considering simply connected neighbourhoods of a given point on the manifold instead.}, Hausdorff and $C^{\infty}$ manifold and $\mathbf{g}$ is a $C^{m}$ metric tensor field on it that is a non degenerate, covariant tensor field of order 2, with the property that at each point of $\mathcal{M}$ one can choose a frame of $n$ real vectors $\{\e_{1},\ldots,\e_{n}\}$, such that $\mathbf{g}(\e_{A},\e_{B})=\eta_{AB}$ where $\boldsymbol{\eta}$ (called \emph{frame metric}) is a (possibly constant) symmetric matrix with prescribed signature.\\
The totality of the sets $\{\e_{A}\}$ (i.e., the sets for every point on the manifold) determines the $GL(n,\mathbb{R})$ frame bundle over $\mathcal{M}$ and defines the tangent bundle $T(\mathcal{M})$ of $\mathcal{M}$. Thus, the matrix $\boldsymbol{\eta}$ simply reflects the inner products of the vectors in the tangent bundle.\\
Another fundamental notion is that of the cotangent bundle $T^{*}(\mathcal{M})$ of $\mathcal{M}$ which, as a linear vector space, is the dual to $T(\mathcal{M})$.
Indeed, if $\{\Th^{A}\}$ denotes the basis of the cotangent space at a point on the manifold, then in a similar manner, the totality of the sets $\{\Th^{A}\}$ (i.e., the sets for every point on the manifold) determines the $GL(n,\mathbb{R})$ coframe bundle over $\mathcal{M}$ and defines the cotangent bundle $T^{*}(\mathcal{M})$ of $\mathcal{M}$.\\
 The duality relation is realised through a linear operation called \emph{contraction} ($\lrcorner$):
\be
\e_{A}\lrcorner\Th^{B}=\delta_{A}^{\phantom{1}B}
\ee
where $\delta_{A}^{\phantom{1}B}$ is the Kronecker delta.

\vspace{0.3cm}
\'{E}lie Cartan has given a formulation of Riemannian geometry in terms of some basic $p$-forms (which are totally antisymmetric $\binom{0}{p}$ tensors; by definition, the coframe vectors are 1-forms) and the four basic operations acting upon them: the \emph{wedge product} ($\wedge$), the \emph{exterior differentiation} ($d$), the contraction ($\lrcorner$) with a frame vector (field) $\e$ and the Lie derivative ($\pounds_{\e}$) with respect to a frame vector (field) $\e$.\\
Some very important properties are (for any $p$-form $\Al$, any $q$-form $\Be$, and any frame vector(s) $\e$):
\bse
\begin{align}
& \e\lrcorner(\Al\wedge\Be)=(\e\lrcorner\Al)\wedge\Be+(-1)^{p}\Al\wedge(\e\lrcorner\Be)\\
& d(\Al\wedge\Be)=(d\Al)\wedge\Be+(-1)^{p}\Al\wedge(d\Be)\\
& \pounds_{\e}d\Al=d\pounds_{\e}\Al\\
& \pounds_{\e}\Al=\e\lrcorner d\Al+d(\e\lrcorner\Al)\\
& \pounds_{\e}(\Al\wedge\Be)=(\pounds_{\e}\Al)\wedge\Be+\Al\wedge(\pounds_{\e}\Be)\\
& \pounds_{\e_{1}}(\e_{2}\lrcorner\Al)=(\pounds_{\e_{1}}\e_{2})\lrcorner\Al+\e_{2}\lrcorner(\pounds_{\e_{1}}\Al)
\end{align}
\ese

\noindent
Returning to the description of Riemannian geometry, the  necessary equations are provided by \emph{\'{E}. Cartan's structure equations} (CSEs):
\begin{itemize}
\item[(A)] The first set, when the \emph{torsion} vanishes, reads:
\bse
\be
d\Th^{A}+\Ga^{A}_{\phantom{1}B}\wedge\Th^{B}=0
\ee
where the $\Ga^{A}_{\phantom{1}B}$ define the \emph{connexion} 1-forms:
\be
\Ga^{A}_{\phantom{1}B}=\gamma^{A}_{\phantom{1}BM}\Th^{M}
\ee
\ese
and the quantities $\gamma^{A}_{\phantom{1}BM}$ are called \emph{Ricci rotation coefficients} (or \emph{spin connexion coefficients} within the context of Newman-Penrose formalism).
\item[(B)] The second set reads:
\bse
\be
\Om^{A}_{\phantom{1}B}=d \Ga^{A}_{\phantom{1}B}+\Ga^{A}_{\phantom{1}M}\wedge\Ga^{M}_{\phantom{1}B}
\ee
where the $\Om^{A}_{\phantom{1}B}$ define the \emph{curvature} 2-forms:
\be
\Om^{A}_{\phantom{1}B}=\frac{1}{2}R^{A}_{\phantom{1}BMN}\Th^{M}\wedge\Th^{N}
\ee
\ese
and the quantities $R^{A}_{\phantom{1}BMN}$ define the Riemann tensor.
\item[(C)] Taking into account the \emph{exterior differential calculus} (EDC), the action of a $d$ upon the first set results in the third set:
\bse
\be
(d\Ga^{A}_{\phantom{1}B}+\Ga^{A}_{\phantom{1}N}\wedge\Ga^{N}_{\phantom{1}B})\wedge\Th^{B}=0
\ee
or, by virtue of the second set:
\be
R^{A}_{\phantom{1}BMN}\Th^{B}\wedge\Th^{M}\wedge\Th^{N}=0
\ee
\ese
which are nothing but the \emph{Jacobi identities}\footnote{Or \emph{algebraic Bianchi identities}, or \emph{cyclic identities}.}.
\item[(D)] Taking into account the EDC, the action of a $d$ upon the second set results in, by virtue of the second set itself, the fourth set:
\be
d\Om^{A}_{\phantom{1}B}=\Om^{A}_{\phantom{1}M}\wedge\Ga^{M}_{\phantom{1}B}-\Ga^{A}_{\phantom{1}M}\wedge\Om^{M}_{\phantom{1}B}
\ee
which are nothing but the \emph{Bianchi identities}.
\item[(E)] The frame metric $\boldsymbol{\eta}$, which is used to raise and lower (co)frame indices, and is subject to the \emph{metricity condition}:
\be
d\eta_{AB}=\Ga_{AB}+\Ga_{BA}
\ee
It is also used to define the \emph{first fundamental form}, $g$:
\be
g=\eta_{AB}\Th^{A}\otimes\Th^{B}
\ee
\end{itemize}
At this point, a comment is deemed necessary. If:
\bse\label{ForObservation1}
\be
\pounds_{\e_{A}}\e_{B}\equiv[\e_{A},\e_{B}]=D^{M}_{\phantom{1}AB}\e_{M}
\ee
where the quantities $D^{A}_{\phantom{1}BM}$ are called \emph{structure functions}, then, by using the basic properties given a few lines back, it can easily be proven that:
\be
d\Th^{A}=-\frac{1}{2}D^{A}_{\phantom{1}MN}\Th^{M}\wedge\Th^{N}=0
\ee
\ese
i.e., $\gamma^{A}_{\phantom{1}[BM]}\equiv\frac{1}{2}(\gamma^{A}_{\phantom{1}BM}-\gamma^{A}_{\phantom{1}MB}) =-\frac{1}{2}D^{A}_{\phantom{1}BM}$

\vspace{0.3cm}
This formalism provides a very elegant, powerful, and simple formulation not only for the isometry condition but also for its integrability conditions. In fact, this formulation becomes even simpler when the frame metric is constant on its entire domain of definition (\emph{rigid frame metric}). Indeed, following \cite{Chinea}, the isometry condition:
\be
\pounds_{\Ze}\mathbf{g}=0
\ee
is completely equivalent to the statement:
\be\label{LieDraggingFForm}
\pounds_{\Ze}(\eta_{AB}\Th^{A}\otimes\Th^{B})=\eta_{AB}\big((\pounds_{\Ze}\Th^{A})\otimes\Th^{B}+\Th^{A}\otimes(\pounds_{\Ze}\Th^{B})\big)=0
\ee
since $d\eta_{AB}=0$; i.e., the Lie dragging, with respect to a KVF $\Ze$, of the first fundamental form vanishes. Because the set of coframe vectors $\{\Th^{A}\}$ constitutes a basis in the cotangent space, it follows that:
\be
\pounds_{\Ze}\Th^{A}=F^{A}_{\phantom{1}B}\Th^{B},\quad\text{in general:}\quad dF^{A}_{\phantom{1}B}\neq 0
\ee
for some non constant quantities $F^{A}_{\phantom{1}B}$; essentialy these are related to the KVF $\Ze$, and thus are to be determined . Upon substitution of this allocation to \eqref{LieDraggingFForm}, it is:
\be\label{EquivalentIsometry}
\pounds_{\Ze}\mathbf{g}=0\Leftrightarrow
\left\{\begin{aligned}
& \pounds_{\Ze}\Th^{A}=F^{A}_{\phantom{1}B}\Th^{B}\\
& \eta_{AM}F^{M}_{\phantom{1}B}+\eta_{MB}F^{M}_{\phantom{1}A}=0
\end{aligned}
\right.
\ee
Once again, taking into account the EDC and the CSEs, the action of a $d$ upon \eqref{EquivalentIsometry} results in the primary integrability condition:
\be\label{FistICforIsometry}
\pounds_{\Ze}\Ga^{A}_{\phantom{1}B}=F^{A}_{\phantom{1}M}\Ga^{M}_{\phantom{1}B}-\Ga^{A}_{\phantom{1}M}F^{M}_{\phantom{1}B}-dF^{A}_{\phantom{1}B}
\ee
In a similar manner, taking into account the EDC and the CSEs, the action of a $d$ upon \eqref{FistICforIsometry} results in the secondary integrability condition:
\be\label{SecondICforIsometry}
\pounds_{\Ze}\Om^{A}_{\phantom{1}B}=F^{A}_{\phantom{1}M}\Om^{M}_{\phantom{1}B}-\Om^{A}_{\phantom{1}M}F^{M}_{\phantom{1}B}
\ee
The integrability conditions of \eqref{SecondICforIsometry} are empty, by virtue of the Bianchi identities.\\
In summary:
\be
\left.\begin{aligned}
& \pounds_{\Ze}\mathbf{g}=0\\
& \text{plus}\\
& \text{Integrability}\\
& \text{Conditions}
\end{aligned}\right\}\Leftrightarrow 
\left\{ 
\begin{aligned}
& g=\eta_{AB}\Th^{A}\otimes\Th^{B}\\
& d\eta_{AB}=0\\
& \pounds_{\Ze}\Th^{A}=F^{A}_{\phantom{1}B}\Th^{B},\quad\text{in general:}\quad dF^{A}_{\phantom{1}B}\neq 0\\
& \eta_{AM}F^{M}_{\phantom{1}B}+\eta_{MB}F^{M}_{\phantom{1}A}=0\\
& \pounds_{\Ze}\Ga^{A}_{\phantom{1}B}=F^{A}_{\phantom{1}M}\Ga^{M}_{\phantom{1}B}-\Ga^{A}_{\phantom{1}M}F^{M}_{\phantom{1}B}-dF^{A}_{\phantom{1}B}\\
& \pounds_{\Ze}\Om^{A}_{\phantom{1}B}=F^{A}_{\phantom{1}M}\Om^{M}_{\phantom{1}B}-\Om^{A}_{\phantom{1}M}F^{M}_{\phantom{1}B}\\
& \text{plus}\\
& \text{CSEs}
\end{aligned}
\right.
\ee


\section{LHS and induced KVFs: the full symmetry group}

The starting poing is the following:
\begin{defn}
Let a structure $(\mathcal{M},\mathbf{g},G_{r})$ be such that:  
\begin{itemize}
\item[$H_{1}$] $\mathcal{M}$ is an $n$ dimensional, simply connected, Hausdorff and $C^{\infty}$ manifold,
\item[$H_{2}$] $\mathbf{g}$ is a $C^{m}$ metric tensor field, defined on the entire manifold that is a non degenerate, covariant tensor field of order 2, with the property that at each point of 
                         $\mathcal{M}$ one can choose a frame of $n$ real vectors $\{\e_{1},\ldots,\e_{n}\}$, such that $\mathbf{g}(\e_{A},\e_{B})=h_{AB}$ where $h_{AB}$ is a (possibly constant)            
                         symmetric matrix of Euclidean signature,
\item[$H_{3}$] $G_{r}$ is a local Lie group of transformations (associated with a Lie algebra $\mathfrak{g}_{r}$) acting simply transitively on $\mathcal{M}$. Therefore, $n=r$ and a bijective 
                         mapping between the set of group parameters and (at least) some neighbourhoods of points on the manifold can be established.
\end{itemize}
Then this structure defines an $n$ dimensional locally homogeneous space (for rigorous accounts see \cite{ContinuousGroupsOfTransoformations}, and for generalisations like \emph{curvature homogeneous spaces}, see \cite{Homogeneity}).
\end{defn}

First, let $U_{p}$ be an open neighbourhood of a given, albeit arbitrary, point $p \in \mathcal{M}$, such that two conditions are met: (a)  the set of group parameters corresponding to the identity group element is a proper subset of $U_{p}$, and (b) if the Lie algebra $\mathfrak{g}_{n}$, as a (real) linear vector space, is spanned by the set of KVFs $\{\Ksi_{A}\}$, then none of its members has singular points on $U_{p}$.\\
The set $\{\Ksi_{A}\}$ can serve as the frame vectors on $\bigcup_{q \in U_{p}}T_{q}(\mathcal{M})$; consequently, the set comprised of their duals, say, $\{\boldsymbol{\phi}^{A}\}$:
\be\label{FirstDuality}
\Ksi_{A}\lrcorner\boldsymbol{\phi}^{B}=\delta_{A}^{\phantom{1}B}
\ee
can also serve as the coframe vectors on  $\bigcup_{q \in U_{p}}T^{*}_{q}(\mathcal{M})$, and it can be used to describe the local Riemannian geometry along the lines of \'{E}. Cartan's theory. Now, since:
\be\label{LieAlgebra}
\forall\quad\Ksi_{A},\Ksi_{B}\in \mathfrak{g}_{n}: [\Ksi_{A},\Ksi_{B}]=C^{M}_{\phantom{1}AB}\Ksi_{M}
\ee
where $C^{M}_{\phantom{1}AB}$ are the strcuture constants corresponding to the Lie algebra $\mathfrak{g}_{n}$, a Lie differentiation with respect to $\Ksi_{A}$ of the duality relation \eqref{FirstDuality} results in:
\be
\pounds_{\Ksi_{A}}\boldsymbol{\phi}^{B}=-C^{B}_{\phantom{1}AM}\boldsymbol{\phi}^{M},\quad\forall\quad A, B \in \{1,\ldots,n\}
\ee
and thus, a first fundamental form like:
\be
g=H_{AB}\boldsymbol{\phi}^{A}\otimes\boldsymbol{\phi}^{B}
\ee
would not \underline{trivially} solve the symmetry condition:
\be
\pounds_{\Ksi_{A}}\mathbf{g}=0
\ee
but, rather, there would be differential constraints upon the (non constant) $H_{AB}$. Therefore, such an adoption would not offer much towards the search for an irreducible (i.e., without spurious degrees of freedom) metric tensor field and with all its symmetries known, simply because one would face a similar, to the initial, problem: that of solving a system of PDEs for $H_{AB}$.\\
A simplification towards the solution to this problem is provided by the basis of the algebra $\widetilde{\mathfrak{g}}_{n}$ of the reciprocal ---to $G_{n}$--- local Lie group of transformations $\widetilde{G}_{n}$, spanned by, say, $\{\X_{A}\}$ and having the defining property\footnote{See the first of references \cite{ContinuousGroupsOfTransoformations}.}:
\bse
\begin{align}
(\forall\quad\Ksi_{A}\in \mathfrak{g}_{n})\wedge (\forall\quad\X_{B}\in \widetilde{\mathfrak{g}}_{n})& : \pounds_{\Ksi_{A}}\X_{B}\equiv[\Ksi_{A},\X_{B}]=0\label{InvariantFrameA}\\
\forall\quad\X_{A},\X_{B}\in \widetilde{\mathfrak{g}}_{n}& : [\X_{A},\X_{B}]=-C^{M}_{\phantom{1}AB}\X_{M}\label{InvariantFrameB}
\end{align}
\ese
where \eqref{InvariantFrameB} reflect the initial conditions (say $\{\X_{A}|_{q\in U_{p}}=\X^{0}_{A}\}$) needed for a solution to the system of PDEs \eqref{InvariantFrameA}.\\
At this point, an assumption is needed; the domain of definition $S_{X}$, determined by the solution to the system of PDEs \eqref{InvariantFrameA}, is supposed to be such that: (a) no member of the set $\{\X_{A}\}$ has singular points on $S_{X}$, and (b) the set of group parameters corresponding to the identity group element is a proper subset of the set $\big(\bigcup_{q \in U_{p}}T_{q}(\mathcal{M})\big)\bigcap S_{X}\equiv \mathcal{K}$.\\
Under the previous assumption, the set $\{\X_{A}\}$ can now serve as the frame vectors on $\bigcup_{q \in \mathcal{K}}T_{q}(\mathcal{M})$,  while the set comprised of their duals, say, $\{\Si^{A}\}$:
\be\label{Duality}
\X_{A}\lrcorner\Si^{B}=\delta_{A}^{\phantom{1}B}
\ee
which is also characterised by the property:
\be\label{KsiLieDraggingSigma}
\pounds_{\Ksi_{A}}\Si^{B}=0,\quad\forall\quad A, B \in \{1,\ldots,n\}
\ee
following immediately by a Lie differentiation with respect to $\Ksi_{A}$, of \eqref{Duality} and implementation of \eqref{InvariantFrameA}, can serve as the coframe vectors on  $\bigcup_{q \in \mathcal{K}}T^{*}_{q}(\mathcal{M})$. Moreover, given the fact that all the inner products amongst the frame vectors $\{\X_{A}\}$ are, by virtue of \eqref{InvariantFrameA}, constant (or dependent on outer\footnote{Here, \emph{outer} stands for variables irrelevant to $\mathcal{M}$ and without any prejudice regarding their character.} variable(s)), it is deduced that the frame metric is also constant (or dependent on outer variable(s)).\\
Thus, a first fundamental form like:
\be\label{RedFundamentalForm}
g_{\text{red.}}=h_{AB}\Si^{A}\otimes\Si^{B}
\ee
defines a smooth metric tensor field throughout $\bigcup_{q \in \mathcal{K}}T^{*}_{q}(\mathcal{M})$, which \underline{trivially} satisfies, the equivalent to, the isometry condition:
\bse
\begin{align}
\pounds_{\Ksi_{A}}g_{\text{red.}} & =h_{MN}\big((\pounds_{\Ksi_{A}}\Si^{M})\otimes\Si^{N}+\Si^{M}\otimes(\pounds_{\Ksi_{A}}\Si^{N})\big)\stackrel{\eqref{KsiLieDraggingSigma}}{=}0\\
d h_{AB} & =0
\end{align}
\ese
rendering the homogeneity manifest.

Both problems, i.e., that of an irreducible form for the compatible metric tensor field $\mathbf{g}$, and that of the knowledge of the full symmetry group of the latter, have been reduced to the search for an irreducible form for the frame metric $\mathbf{h}$ and its symmetries. 


\subsection{Irreducible forms for the frame metric tensor field in $n$-dimensional LHS}

In the calculus on differentiable manifolds, there are two distinct categories\footnote{Here, the word \emph{categories} stands for a class or division of  things regarded as having particular shared characteristics; a class which might have a structure like that of a group.} of gauge freedom. Both categories can be given the structure of a (local) continuous group  --along, with their disconnected components, corresponding to discrete symmetries\footnote{See: \cite{Stephani_et_al.}, the last two of references \cite{ContinuousGroupsOfTransoformations} and, of course, \cite{Cartan}.}:
\begin{itemize}
\item[$GF_{1}$] The category related to changes in local coordinate systems in a given atlas on the manifold, constituting the \emph{diffeomorphisms group}.
\item[$GF_{2}$] The category related to changes in basis in the tangent space, i.e., the group associated with the (co)tangent bundle. This could also be thought of as being the symmetry group of the         
                           CSEs\footnote{Irrespectively of a torsion and/or a metric tensor field on the tangent space.}. Further, there is always a residual, yet trivial, freedom: that reflected in the action of the group 
                           $GL(n,\mathbb{R})$.
\end{itemize}
Although a rare phenomenon, nothing prevents these two categories from having common members; frame basis changes inducing diffeomorphisms might exist.\\
If $O(\mathbf{h}) $ and $\mathfrak{o}(\mathbf{h})$ are a Lie group and its Lie algebra respectively\footnote{Of order $n(n-1)/2$. The disconnected, to the identity, components of the initial continuous group have been neglected.}, described in $GF_{2}$ category:
\bse
\begin{align}
O(\mathbf{h}) & =\{Q^{A}_{\phantom{1}B}: h_{AB}Q^{A}_{\phantom{1}M}Q^{B}_{\phantom{1}N}=h_{MN}\},\quad\text{in general:}\quad dQ^{A}_{\phantom{1}B}\neq 0\\
\mathfrak{o}(\mathbf{h}) & =\{q^{A}_{\phantom{1}B}: h_{AM}q^{M}_{\phantom{1}B}+h_{MB}q^{M}_{\phantom{1}A}=0\},\quad\text{in general:}\quad dq^{A}_{\phantom{1}B}\neq 0
\end{align}
\ese
then it is clear that an irreducible frame metric is presupposed, for spurious degrees of freedom in $\mathbf{h}$ might lead to a different group.

Apparently, category $GF_{2}$ ---by its definition, as a symmetry of the metric tensor field $\mathbf{h}$--- should  not contribute to the search for an irreducible form for the latter. On the other hand, the residual freedom provided by $GL(n,\mathbb{R})$, as it stands, does not suffice; not only it does not cover the case where $\mathbf{h}$ might depend on irrelevant parameters, but it might also lead to wrong results. For instance, since $SO(n_{1},n-n_{1},\mathbb{R})\triangleleft GL(n,\mathbb{R})$ (i.e., a subgroup), rotations can be used to diagonalise any matrix $h_{AB}$ with signature $(n_{1},n-n_{1})$, but this is a general result only when the homogeneity corresponds to an abelian group --as it is well known. The solution towards the search for an irreducible form for the frame metric in an $n$-dimensional LHS is provided by the following:
\begin{pro}
Let a LHS be characterised by a Lie algebra \eqref{LieAlgebra}, along with a fundamental form:
\be
g_{\text{red.}}=h_{AB}\Si^{A}\otimes\Si^{B}
\ee
defined throughout $\bigcup_{q \in \mathcal{K}}T^{*}_{q}(\mathcal{M})$. Then, \underline{a class of irreducible frame metrics} is given, in infinitesimal form, by the family:
\be
\{h_{AB}+h_{AM}C^{M}_{\phantom{1}NB}\varepsilon^{N}+h_{MB}C^{M}_{\phantom{1}NA}\varepsilon^{N}\}
\ee
where the set $\{\varepsilon^{N}\}$ is a collection of, at most $n$, functions depending on the outer variables of $\mathbf{h}$ ---if any--- multiplied by parameters close to zero.
\end{pro}
\begin{proof}
Let $\mathbb{F}$ be the set:
\be
\mathbb{F}=\left\{ 
\begin{aligned}
(S^{A}_{\phantom{1}B}: U_{AB} & \rightarrow Y_{AB},\quad  U_{AB}, Y_{AB} \subseteq \mathbb{R},\ \text{analytic}~\forall\ A, B)\\
& \wedge\\
(|S^{A}_{\phantom{1}B}| & \neq 0)\\
& \wedge\\
(S^{A}_{\phantom{1}B}C^{B}_{\phantom{1}MN} & =C^{A}_{\phantom{1}KL}S^{K}_{\phantom{1}M}S^{L}_{\phantom{1}N})
\end{aligned}
\right\}
\ee
i.e., the set of all, structure constants form preserving, invertible matrices, with each component of which defining a different, analytic mapping from a subset of $\mathbb{R}$ to another subset. In total, $n^{2}$ analytic mappings and two families $\{U_{AB}\}$ and $\{Y_{AB}\}$, with $n^{2}$ members each,  of subsets of $\mathbb{R}$ are needed.\\
It should be obvious that $\mathbb{F}$ defines multi parametric subsets of $GL(n,\mathbb{R})$: each for every $n^{2}$-tuple of points, say $\{w^{\alpha}\}_{\alpha \in \{1,\ldots,n^{2}\}}$, on the family of sets $\{U_{AB}\}$. This set, endowed with the usual operations, can be given the structure of a local Lie group; its corresponding Lie algebra $\mathfrak{F}$ will be spanned by a subspace of $\mathfrak{gl}(n,\mathbb{R})$.\\
A trivial calculation proves that the set:
\be
\{\widetilde{\X}_{B}\}=\{\X_{A}(S^{-1})^{A}_{\phantom{1}B},\ S^{A}_{\phantom{1}B} \in \mathbb{F}\}
\ee
and only that, preserves the form of the system of PDEs \eqref{InvariantFrameA} along with the (consequences of its) initial conditions, i.e., \eqref{InvariantFrameB}, thus constituting a Lie-point symmetry\footnote{As within the context of differential equations --see the 7th from references \cite{ContinuousGroupsOfTransoformations}.} of this system. The duality expressed in \eqref{Duality} induces exactly the same symmetry on the corresponding cotangent space; the set: 
\be\label{AutIndDiffeo}
\{\widetilde{\Si}^{A}\}=\{S^{A}_{\phantom{1}B}\Si^{B},\ S^{A}_{\phantom{1}B} \in \mathbb{F}\}
\ee
preserves the manifest homogeneity of the first fundamental form:
\be
\left.\begin{aligned}
& g_{\text{red.}}=h_{AB}\Si^{A}\otimes\Si^{B}\\
& g_{\text{red.}}=\widetilde{h}_{MN}\widetilde{\Si}^{M}\otimes\widetilde{\Si}^{N}
\end{aligned}
\right\}\stackrel{\eqref{AutIndDiffeo}}{\Rightarrow} \widetilde{h}_{MN}=h_{AB}(S^{-1})^{A}_{\phantom{1}M}(S^{-1})^{B}_{\phantom{1}N}\label{EquivalentFrameMetrics}
\ee
It is trivial to prove\footnote{The proof of that statement is trivial; one has only to consider the CSEs corresponding to two first fundamental forms, $g$ and $\widetilde{g}$ in terms of the same coframe vectors, with frame metrics related as in \eqref{EquivalentFrameMetrics}, and the defining property of the matrices $S^{A}_{\phantom{1}B}$, i.e., that of the structure constants form preservation.} that this relation defines an equivalence class for frame metrics; i.e., $h_{AB}$ and $\widetilde{h}_{AB}$ determine the same local Riemannian geometry --yet a redundancy, hidden in the matrices $S^{A}_{\phantom{1}B}$, might (dis)appear at will.

Next, it would be most useful to find the Lie algebra $\mathfrak{F}$, and for this scope it is necessary to consider those matrices which are connected to the identity:
\be
S^{A}_{\phantom{1}B}(0)=\delta^{A}_{\phantom{1}B}
\ee
Then, substitution to:
\be
S^{A}_{\phantom{1}B}(w^{\alpha})C^{B}_{\phantom{1}MN}=C^{A}_{\phantom{1}KL}S^{K}_{\phantom{1}M}(w^{\alpha})S^{L}_{\phantom{1}N}(w^{\alpha})
\ee
and differentiation with respect to each particular parameter $w^{\alpha_{i}}$ at zero, results in:
\bse\label{AutGenerators}
\begin{align}
&\lambda^{A}_{\phantom{1}(\alpha_{i})B}C^{B}_{\phantom{1}MN}=\lambda^{Q}_{\phantom{1}(\alpha_{i})M}C^{A}_{\phantom{1}QN}+\lambda^{Q}_{\phantom{1}(\alpha_{i})N}C^{A}_{\phantom{1}MQ}\\
&\text{where:}\notag\\
&\lambda^{A}_{\phantom{1}(\alpha_{i})B}\equiv\frac{dS^{A}_{\phantom{1}B}(w^{\alpha})}{d w^{\alpha_{i}}}\Big|_{w^{\alpha}=0},\quad \alpha_{i} \in \{1,\ldots,n^{2}\}
\end{align}
\ese
are the requested generators. The number of independent solutions to this linear system determines the number of the independent parameters (i.e., the upper bound for $\alpha_{i}$ --say $q\leq n^{2}$), and thus the number of the generators of $\mathfrak{F}$. Essentially, $\mathbb{F}$ defines the \emph{automorphism group} corresponding to $G_{n}$. Due to the Jacobi identities, the space of solutions to \eqref{AutGenerators} can be written as:
\be
\{\lambda^{A}_{\phantom{1}(\alpha_{i})B}\}=\{C^{A}_{\phantom{1}(a_{I})B}, E^{A}_{\phantom{1}(\alpha_{J})B}\},\quad \left\{ \begin{aligned}
& a_{I} \in \{1,\ldots,n\}\\
& \alpha_{J} \in \{n+1,\ldots,q\}
\end{aligned}\right.
\ee
The subset of $\{C^{A}_{\phantom{1}(a_{I})B}\}$, as generators, defines a Lie algebra corresponding to the \emph{inner automorphism} group; a subgroup of the automorphism group.

Since the interest is focused on the infinitesimal form for the transformation of the frame metrics \eqref{EquivalentFrameMetrics}, it suffices to consider automorphic group elements near the identity:
\be\label{InfinitesimalAut}
S^{A}_{\phantom{1}B}(w^{\alpha}) \simeq \delta^{A}_{\phantom{1}B}+w^{\alpha_{i}}\lambda^{A}_{\phantom{1}(\alpha_{i})B},\quad w^{\alpha_{i}}\rightarrow 0
\ee
Then, condition \eqref{EquivalentFrameMetrics} becomes (up to the first order):
\be\label{TowardsTheFinalClass}
\widetilde{h}_{AB}\simeq h_{AB}-h_{AM}\lambda^{M}_{\phantom{1}(\alpha_{i})B}w^{\alpha_{i}}-h_{MB}\lambda^{M}_{\phantom{1}(\alpha_{i})A}w^{\alpha_{i}},\quad w^{\alpha_{i}}\rightarrow 0
\ee
At this point, it is necessary to remember that these changes upon the form of the frame metric can be divided into two categories: those which are induced by diffeomorphisms, and those which are symmetries of the frame metric (i.e., they leave it form invariant as a Lorentz transformation leaves form invariant the Minkowski metric on the tangent space). Since an irreducible form for $h_{AB}$ is not known (as this is the desired result), it is not possible ---at this stage--- to identify which part(s) of \eqref{TowardsTheFinalClass} belong to the first and which to the second category. But, the first category must be induced by equally infinitesimal transformations.

It is, therefore, deemed appropriate to find which generators, contributing to \eqref{InfinitesimalAut}, are induced by infinitesimal GCTs.\\
By construction, the set $\bigcup_{q \in \mathcal{K}}T_{q}(\mathcal{M})$ not only contains the set $\{X_{A}\}$, without any singular points, but also constitutes a real vector space. Thus, the generator of any desired infinitesimal GCT will be a linear combination of those vector fields. There are two possibilities for this linear combination: having either constant or non constant coefficients. The first possibility not only ensures, in principle, the existence of well defined vector fields, throughout $\bigcup_{q \in \mathcal{K}}T_{q}(\mathcal{M})$, but also preserves the form invariance of the initially adopted KVFs; indeed, the Lie derivative of the KVFs with respect to any generator of the desired GCTs vanishes if and only if the coefficients in the linear combination are constant. The second possibility, which corresponds to a kind of self similarity of the initially adopted KVFs, leads to a more rich structure. Yet, this second possibility is more restrictive in the following sense: a linear combination with non constant coefficients might not be well defined everywhere. Also, the resulted freedom, being of local character (with respect to $\bigcup_{q \in \mathcal{K}}T_{q}(\mathcal{M})$), can not be extended in many cases --like that of embedding to higher dimensions. For all these reasons, the first possibility ---which obviously constitutes a minimum requirement--- is adopted.\\
Let $\boldsymbol{\Pi}_{\alpha_{i}}$ be a member of a class of vector fields operating as generators for a family of infinitesimal (i.e., multi parametric) GCTs:
\bse
\begin{align}
& \boldsymbol{\Pi}_{\alpha_{i}}=\Sigma^{A}_{\phantom{1}\alpha_{i}}\X_{A},\quad d\Sigma^{A}_{\phantom{1}\alpha_{i}}=0
~(d\Sigma^{A}_{\phantom{1}\alpha_{i}}=0\stackrel{\eqref{InvariantFrameA}}{\iff} \pounds_{ \boldsymbol{\Pi}_{\alpha_{i}}}\Ksi_{A}=0)\label{InfinitesimalGCTA}\\
& \{x^{a}\}\rightarrow \{\widetilde{x}^{a}\}: \widetilde{x}^{a}+w^{\alpha_{i}}\Pi^{a}_{\phantom{1}\alpha_{i}}(x^{b}),\quad w^{\alpha_{i}}\rightarrow 0\label{InfinitesimalGCTB}
\end{align}
\ese
The quantities $\Sigma^{A}_{\phantom{1}\alpha_{i}}$ may (and will) depend on those outer parameters upon which $\mathbf{h}$ is dependent.
Before continuing, a simple observation: by \eqref{InfinitesimalAut}, the infinitesimal version of  \eqref{AutIndDiffeo} can easily be read off:
\be
\widetilde{\Si}^{A}\simeq \Si^{A}+w^{\alpha_{i}}\lambda^{A}_{\phantom{1}(\alpha_{i})B}\Si^{B},\quad w^{\alpha_{i}}\rightarrow 0
\ee
Thus, if the change in form of the coframe vectors is supposed to be induced by \eqref{InfinitesimalGCTB}, then the condition:
\be\label{ForTheExistenceOfGCT}
\pounds_{\boldsymbol{\Pi}_{\alpha_{i}}}\Si^{A}=\lambda^{A}_{\phantom{1}(\alpha_{i})B}\Si^{B}
\ee
must hold on $\bigcup_{q \in \mathcal{K}}T^{*}_{q}(\mathcal{M})$. By virtue of the Jacobi identities, the integrability condition of \eqref{ForTheExistenceOfGCT} is empty, thus this system always admits a well defined solution of the form \eqref{InfinitesimalGCTB}. Therefore, the setting of the system \eqref{InfinitesimalGCTB} and \eqref{ForTheExistenceOfGCT} is well posed and always admits a solution, because of both the definition for the set $\bigcup_{q \in \mathcal{K}}T^{*}_{q}(\mathcal{M})$ and its (empty) integrability conditions.\\
Using, \eqref{InfinitesimalGCTA}, \eqref{Duality}, and a few standard properties of the Lie derivative\footnote{See, Section 2.}, it is inferred that:
\be
C^{A}_{\phantom{1}MB}\Sigma^{M}_{\phantom{1}\alpha_{i}}=\lambda^{A}_{\phantom{1}\alpha_{i}B},\quad \text{and thus:}\quad \alpha_{i} = a_{I}
\ee
In other words, only the inner automorphisms can be induced by infinitesimal diffeomorphisms --cf.\ \cite{Gauge}. Substitution of the last result to \eqref{TowardsTheFinalClass} and the allocation $\varepsilon^{N}\equiv -\Sigma^{N}_{\phantom{1}a_{I}}w^{a_{I}}$ complete the proof. 
\end{proof}
At this point, it should be stressed, once again, that a non constant linear combination for the generators $\boldsymbol{\Pi}_{\alpha_{i}}$ (i.e., $d\Sigma^{A}_{\phantom{1}\alpha_{i}}\neq 0$) would have led to both inner and outer automorphisms (the \emph{rich structure} mentioned earlier in the last part of the proof), but then one would be in position of implementing that structure only in the case where the study would exclusively be restricted on the homogenous spaces per se, without any ambition to consider them as a part of a more complex system --like when a homogeneous space is embedded into another higher dimensional space.\\
Therefore, from this point of view, the Proposition offers \emph{a class rather than a unique family} of irreducible frame metrics. Although it seems that not the entire ``symmetry (i.e., inner and outer automorphisms) of the symmetry (i.e., Lie group of transformations acting simply transitively)'' is exploited, a use of which would be done at the expense of a restricted study, the results are based on the least minimum requirement and are valid, without any modification, even within other more complex frameworks --see, e.g., the third reference of \cite{Gauge}. 

The Proposition defines a class of geometrically equivalent and irreducible frame metrics --provided that the entire freedom represented by the descriptors $\{\varepsilon^{A}\}$ has been implemented. This result stands on its own right: indeed, it has been proven that a (sub)group of the automorphism group ---that of the inner automorphisms--- not only is induced by the action of another local Lie group ---that of diffeomorphisms--- but also gives a geometric definition of gauge symmetries. Of course, this is not quite a new result; the symmetry of the symmetry is of central importance in many research works. An indicative idea of the general interest might be found in \cite{Gauge} and the references therein. On the other hand though, in many instances, there is a kind of disagreement on what is (or should be) considered as gauge freedom; see the results found in \cite{GlobalandGauge} as they are opposed to those in the first of references \cite{Gauge}. 

In any case, the first goal towards the discovery of the full symmetry group, i.e., the irreducible form for the frame metric, has been achieved. 


\subsection{Induced Symmetry}

Let, again, a LHS be characterised by a Lie algebra \eqref{LieAlgebra}, along with a fundamental form:
\be
g_{\text{irred.}}=h^{\text{irred.}}_{AB}\Si^{A}\otimes\Si^{B}
\ee
defined throughout $\bigcup_{q \in \mathcal{K}}T^{*}_{q}(\mathcal{M})$. It is meant that the freedom provided by the inner automorphisms (sub)group of $G_{n}$ has been implemented in order to cast the frame metric in an irreducible (though not unique) form.

If a metric tensor field $\mathbf{g}_{\text{irred.}}$ admits a KVF(s) $\Ze$, irrelevant to the initial KVFs which describe homogeneity, then the system of the Killing equations:
\be
\pounds_{\Ze}\mathbf{g}_{\text{irred.}}=0
\ee
is completely equivalent to:
\be
\pounds_{\Ze}g_{\text{irred.}}=0
\ee
for the corresponding first fundamental form $g$. Following \cite{Chinea}, this system is, in turn, equivalent to:
\be\label{TowardsFinalSymmetry}
\left.\begin{aligned}
& \pounds_{\Ze}g_{\text{irred.}}=0\\
& \text{plus}\\
& \text{Integrability}\\
& \text{Conditions}
\end{aligned}\right\}\Leftrightarrow 
\left\{ 
\begin{aligned}
& g_{\text{irred.}}=h^{\text{irred.}}_{AB}\Si^{A}\otimes\Si^{B}\\
& dh^{\text{irred.}}_{AB}=0\\
& \pounds_{\Ze}\Si^{A}=\Psi^{A}_{\phantom{1}B}\Si^{B},\quad\text{in general:}\quad d\Psi^{A}_{\phantom{1}B}\neq 0\\
& h^{\text{irred.}}_{AM}\Psi^{M}_{\phantom{1}B}+h^{\text{irred.}}_{MB}\Psi^{M}_{\phantom{1}A}=0\\
& \pounds_{\Ze}\Ga^{A}_{\phantom{1}B}=\Psi^{A}_{\phantom{1}M}\Ga^{M}_{\phantom{1}B}-\Ga^{A}_{\phantom{1}M}\Psi^{M}_{\phantom{1}B}-d\Psi^{A}_{\phantom{1}B}\\
& \pounds_{\Ze}\Om^{A}_{\phantom{1}B}=\Psi^{A}_{\phantom{1}M}\Om^{M}_{\phantom{1}B}-\Om^{A}_{\phantom{1}M}\Psi^{M}_{\phantom{1}B}\\
& \text{plus}\\
& \text{CSEs}
\end{aligned}
\right.
\ee
A further simplification is possible by two observations:
\begin{itemize}
\item[$Ob_{1}$] Because of: \eqref{InvariantFrameB}, \eqref{Duality}, the first set of the CSEs becomes:
                           \be
                           d\Si^{A}=\frac{1}{2}C^{A}_{\phantom{1}MN}\Si^{M}\wedge\Si^{N}
                           \ee
                           also, cf.\  \eqref{ForObservation1}.
\item[$Ob_{2}$] The condition $h^{\text{irred.}}_{AM}\Psi^{M}_{\phantom{1}B}+h^{\text{irred.}}_{MB}\Psi^{M}_{\phantom{1}A}=0$ denotes the antisymmetry, in its indices, of the matrix 
                           $\Psi_{AB}\equiv h^{\text{irred.}}_{AM}\Psi^{M}_{\phantom{1}B}$. Therefore, the latter can be parametrized by an antisymmetric matrix:
                           \bse
                           \begin{align}
                           & h^{\text{irred.}}_{AM}\Psi^{M}_{\phantom{1}B}=F_{AB}\Rightarrow \Psi^{A}_{\phantom{1}B}=h_{\text{irred.}}^{AM}F_{MB}\\
                           & F_{(AB)}=0
                           \end{align}
                           \ese      
\end{itemize}
Using these observations, \eqref{TowardsFinalSymmetry} assumes its simplest form:
\be\label{FinalSymmetry}
\left\{ 
\begin{aligned}
& g_{\text{irred.}}=h^{\text{irred.}}_{AB}\Si^{A}\otimes\Si^{B}\\
& dh^{\text{irred.}}_{AB}=0\\
& \pounds_{\Ze}\Si^{A}=h_{\text{irred.}}^{AN}F_{NB}\Si^{B},\quad\text{in general:}\quad dF_{AB}\neq 0\\
& F_{(AB)}=0\\
& d\Si^{A}=\frac{1}{2}C^{A}_{\phantom{1}MN}\Si^{M}\wedge\Si^{N}\\
& \Ga^{A}_{\phantom{1}B}=\gamma^{A}_{\phantom{1}BM}\Si^{M}, \gamma^{A}_{\phantom{1}[BM]}=\frac{1}{2}C^{A}_{\phantom{1}BM}\\
& h^{\text{irred.}}_{AS}\gamma^{S}_{\phantom{1}BM}+h^{\text{irred.}}_{SB}\gamma^{S}_{\phantom{1}AM}=0\\
& \pounds_{\Ze}\Ga^{A}_{\phantom{1}B}=(h_{\text{irred.}}^{AN}F_{NM})\Ga^{M}_{\phantom{1}B}-\Ga^{A}_{\phantom{1}M}(h_{\text{irred.}}^{MN}F_{NB})-h_{\text{irred.}}^{AN}dF_{NB}\\
& \Om^{A}_{\phantom{1}B}=\frac{1}{2}R^{A}_{\phantom{1}BMN}\Si^{M}\wedge\Si^{N}\\
& \Om^{A}_{\phantom{1}B}=d \Ga^{A}_{\phantom{1}B}+\Ga^{A}_{\phantom{1}M}\wedge\Ga^{M}_{\phantom{1}B}\\
& \pounds_{\Ze}\Om^{A}_{\phantom{1}B}=(h_{\text{irred.}}^{AN}F_{NM})\Om^{M}_{\phantom{1}B}-\Om^{A}_{\phantom{1}M}(h_{\text{irred.}}^{MN}F_{NB})
\end{aligned}
\right\}
\ee
The merits of this approach are both simplicity and clarity. Simplicity, because the last integrability condition imposes algebraic constraints upon the matrix $F_{AB}$ --something which renders the solving process, regarding the other steps, easier. Clarity, for the nature of the matrix $F_{AB}$ encodes much information: since it is antisymmetric, only up to $n(n-1)/2$ independent components may exist, each for every extra KVF $\Ze$. Thus, if $F_{AB}=0$ then the only KVFs admitted by the space are those given initially. On the other hand, when the number of components is the maximum, then the total number of the admitted KVFs is: $n$ (those expressing homogeneity)+$n(n-1)/2$ (those corresponding to the matrix $F_{AB}$)=$n(n+1)/2$, i.e., the case of maximally symmetric spaces \cite{ContinuousGroupsOfTransoformations}.


\section{LHS in 3 dimensions, and applications: Bianchi cosmological prototypes and special solutions to EFEs}

\begin{conv}
Lower case Greek indices are used for any coordinate space in $4$ dimensions. It is assumed that all 3 dimensional LHS have metric tensor fields of Euclidean signature: $(+,+,+)$.
\end{conv} 

This section is divided into two parts: in the first part, the analysis developed in the previous section, i.e., the Proposition regarding the irreducible form for the frame metric $\mathbf{h}$ along with the symmetry system \eqref{FinalSymmetry}, is applied when $n=3$. In this case (which is of physical interest), there are nine distinct continuous Lie groups of transformations associated with nine real Lie algebras. These groups and their algebras have been classified by L. Bianchi\footnote{L. Bianchi:\\
1. Mem.\ della Soc.\ Italiana delle Scienze Ser.\ 3a, \textbf{11}, (1897), 267;\\
2. \textit{Lezioni Sulla Teoria Dei Gruppi Continui Finiti Di Transformazioni}, Spoerri, Pisa, 1918} --although S. Lie found them first (see the corresponding reference in \cite{ContinuousGroupsOfTransoformations}). Thus, the \emph{Bianchi Types I-IX} emerged.\\
In the second part, the results found in the first part are used in the search for special solutions to the EFEs. No new ---to the literature--- solutions are obtained, yet this pedantic application carries the ambition to exhibit a simple application of the entire analysis, something which might be of interest in other areas.

\subsection{Bianchi Types}

For every Bianchi Type, the information given includes the structure constants $C^{E}_{\phantom{1}AB}$ of the corresponding Lie algebra, the Killing vector fields $\{\Ksi_{A}\}$, the set of both left invariant fields $\{\X_{A}\}$, and right invariants $\{\Si^{A}\}$ ---as in \cite{RyanShepley}--- and an irreducible form for the frame metric $h^{\text{irred.}}_{AB}$ along with the further KVF(s). Of course, a local coordinate system $\{x^{A}\}=\{x,y,z\}$ has been adopted to express not only all these vector fields with the coordinate basis $\{\partial_{x^{A}}\equiv\partial/\partial x^{A}\}$, but also the dual basis $\{d x^{A}\}$. Also, Bianchi Types VI and VII are each a family of groups parametrized by $q$ within the limits given.\\
Finally, it should be stressed that only the most generic case will be of interest each time. To this end, the initial reducible frame metric will be assumed to be the most general: 
\bdm
h^{\text{red.}}_{AB}=  \begin{pmatrix}
                                        \widetilde{h}_{11} & \widetilde{h}_{12} & \widetilde{h}_{13} \\
                                        \widetilde{h}_{12} & \widetilde{h}_{22} & \widetilde{h}_{23} \\
                                        \widetilde{h}_{13} & \widetilde{h}_{23} & \widetilde{h}_{33}
                                       \end{pmatrix}
\edm
i.e., neither further relations amongst the $\widetilde{h}_{ij}$ nor discriminating cases will be considered. The only restriction is that all the frame metrics $\mathbf{h}$ (i.e., both reducible and irreducible) are supposed to be positive definite. It is obvious, that specializations ---like that corresponding to biaxial symmetry (i.e., $\widetilde{h}_{11}=\widetilde{h}_{22}$), etc.--- on the frame metric components may lead to further symmetries. 

\subsubsection*{Bianchi Type I}
This type is characterised by:
\bse
\begin{align}
C^{E}_{\phantom{1}AB} & = 0 \\
\{\Ksi_{A}\} & = \{\partial_{x}, \partial_{y}, \partial_{z}\}\\
\{\X_{A}\} & = \{\partial_{x}, \partial_{y}, \partial_{z}\}\\
\{\Si^{A}\} & = \{dx, dy, dz\}\\
h^{\text{irred.}}_{AB} & = \begin{pmatrix}
                                              h_{11} & 0          & 0\\
                                              0          & h_{22} & 0 \\
                                              0          & 0           & h_{33}
                                              \end{pmatrix}
\end{align}
\ese
Although all the structure constants vanish, the action of the group $GL(3,\mathbb{R})$ can still be implemented. In fact, the maximum number (i.e., 3) of GCTs can be used to bring the frame metric to a diagonal form; the degrees of freedom left are just the three eigenvalues of the matrix. This is the only case with such a singular behaviour --a characteristic feature of all the abelian prototypes, irrespectively of the dimensions.\\
There are three more KVFs:
\be
\{\Ze_{A}\}= & \{-\frac{z}{h_{22}}\partial_{y}+\frac{y}{h_{33}}\partial_{z}, -\frac{z}{h_{11}}\partial_{x}+\frac{x}{h_{33}}\partial_{z}, -\frac{y}{h_{11}}\partial_{x}+\frac{x}{h_{22}}\partial_{y} \}
\ee
Therefore, Bianchi Type I admits a $G_{6}$ local Lie group of transformations, with the corresponding algebra (only non vanishing commutators are given):
\be\label{Type_I_G6}
\begin{array}{cccc}
&[\Ze_{1},\Ze_{2}]=\frac{1}{h_{33}}\Ze_{3} &[\Ze_{2},\Ze_{3}]=\frac{1}{h_{11}}\Ze_{1} & [\Ze_{3},\Ze_{1}]=\frac{1}{h_{22}}\Ze_{2} \\
& & &  \\
&[\Ksi_{2},\Ze_{1}]=\frac{1}{h_{33}}\Ksi_{3} & [\Ksi_{3},\Ze_{1}]=-\frac{1}{h_{22}}\Ksi_{2} & [\Ksi_{1},\Ze_{2}]=\frac{1}{h_{33}}\Ksi_{3} \\
& & &  \\
&[\Ksi_{3},\Ze_{2}]=-\frac{1}{h_{11}}\Ksi_{1} &[\Ksi_{1},\Ze_{3}]=\frac{1}{h_{22}}\Ksi_{2}  & [\Ksi_{2},\Ze_{3}]=-\frac{1}{h_{11}}\Ksi_{1}
\end{array}
\ee
Of course, the space is maximally symmetric --as expected.

\subsubsection*{Bianchi Type II}
This type is characterised by:
\bse
\begin{align}
C^{1}_{\phantom{1}23} & =1\\
\{\Ksi_{A}\} & = \{\partial_{y}, \partial_{z}, \partial_{x}+z\partial_{y}\}\\
\{\X_{A}\} & = \{\partial_{y}, x\partial_{y}+\partial_{z}, \partial_{x}\}\\
\{\Si^{A}\} & = \{dy-xdz, dz, dx\}\\
h^{\text{irred.}}_{AB} & =  \begin{pmatrix}
                                              h_{11} & 0 & 0 \\
                                              0 & h_{22} & h_{23} \\
                                              0 & h_{23} & h_{33}
                                             \end{pmatrix}
\end{align}
\ese
Only two, out of the available three, GCTs can be used. There is one more KVF:
\be
\Ze=\frac{z h_{22}+x h_{23}}{h_{22}h_{33}-h_{23}^{2}}\partial_{x}+\frac{z^{2} h_{22}-x^{2} h_{33}}{2(h_{22}h_{33}-h_{23}^{2})}\partial_{y}-\frac{z h_{23}+x h_{33}}{h_{22}h_{33}-h_{23}^{2}}\partial{z}
\ee
Therefore, Bianchi Type II admits a $G_{4}$ local Lie group of transformations, with the corresponding algebra (only non vanishing commutators are given):
\be\label{Type_II_G4}
\begin{array}{cc}
&[\Ksi_{2},\Ksi_{3}] = \Ksi_{1}\\
& \\
& [\Ksi_{2},\Ze]=-\frac{h_{23}}{h_{22}h_{33}-h_{23}^{2}}\Ksi_{2}+\frac{h_{22}}{h_{22}h_{33}-h_{23}^{2}}\Ksi_{3}\\
&  \\
&[\Ksi_{3},\Ze]=-\frac{h_{33}}{h_{22}h_{33}-h_{23}^{2}}\Ksi_{2}+\frac{h_{23}}{h_{22}h_{33}-h_{23}^{2}}\Ksi_{3} 
\end{array}
\ee

\subsubsection*{Bianchi Type III}
This type is characterised by:
\bse
\begin{align}
C^{1}_{\phantom{1}13} & = 1\\
\{\Ksi_{A}\} & = \{\partial_{y}, \partial_{z}, \partial_{x}+y\partial_{y}\}\\
\{\X_{A}\} & =  \{e^{x}\partial_{y}, \partial_{z}, \partial_{x}\}\\
\{\Si^{A}\} & = \{e^{-x}dy, dz, dx\}\\
h^{\text{irred.}}_{AB} & = \begin{pmatrix}
                                             h_{11} & h_{12} & 0 \\
                                             h_{12} & h_{22} & 0 \\
                                             0 & 0 & h_{33}
                                             \end{pmatrix}
\end{align}
\ese
Only two, out of the available three, GCTs can be used. There is one more KVF:
\be
\Ze=\frac{y}{h_{33}}\partial_{x}+\Big(\frac{y^{2}}{2h_{33}}-\frac{e^{2x}h_{22}}{2(h_{11}h_{22}-h_{12}^{2})}\Big)\partial_{y}+\frac{e^{x}h_{12}}{h_{11}h_{22}-h_{12}^{2}}\partial_{z}
\ee
Therefore, Bianchi Type III admits a $G_{4}$ local Lie group of transformations, with the corresponding algebra (only non vanishing commutators are given):
\be\label{Type_III_G4}
\begin{array}{cccc}
&[\Ksi_{1},\Ksi_{3}] = \Ksi_{1} & [\Ksi_{1},\Ze]=\frac{1}{h_{33}}\Ksi_{3} & [\Ksi_{3},\Ze]=\Ze
\end{array}
\ee

\subsubsection*{Bianchi Type IV}
This type is characterised by:
\bse
\begin{align}
C^{1}_{\phantom{1}13} & =C^{1}_{\phantom{1}23}=C^{2}_{\phantom{1}23}=1\\
\{\Ksi_{A}\} & = \{\partial_{y}, \partial_{z}, \partial_{x}+(y+z)\partial_{y}+z\partial_{z}\}\\
\{\X_{A}\} & =  \{e^{x}\partial_{y}, xe^{x}\partial_{y}+e^{x}\partial_{z}, \partial_{x}\}\\
\{\Si^{A}\} & = \{e^{-x}dy-xe^{-x}dz, e^{-x}dz, dx\}\\
h^{\text{irred.}}_{AB} & = \begin{pmatrix}
                                              h_{11} & 0          & 0\\
                                              0          & h_{22} & 0 \\
                                              0          & 0           & h_{33}
                                              \end{pmatrix}
\end{align}
\ese
All the three available GCTs can be used to bring the frame metric to a diagonal form; the degrees of freedom left are just the three eigenvalues of the matrix. There is no further KVF, since the integrability conditions imply that $F_{AB}=0$.

\subsubsection*{Bianchi Type V}
This type is characterised by:
\bse
\begin{align}
C^{1}_{\phantom{1}13} & =C^{2}_{\phantom{1}23}=1\\
\{\Ksi_{A}\} & = \{\partial_{y}, \partial_{z}, \partial_{x}+y\partial_{y}+z\partial_{z}\}\\
\{\X_{A}\} & =  \{e^{x}\partial_{y}, e^{x}\partial_{z}, \partial_{x}\}\\
\{\Si^{A}\} & = \{e^{-x}dy, e^{-x}dz, dx\}\\
h^{\text{irred.}}_{AB} & = \begin{pmatrix}
                                              h_{11} & h_{12}          & 0\\
                                              h_{12} & h_{22} & 0 \\
                                              0          & 0           & \sqrt{h_{11}h_{22}-h_{12}^{2}}
                                              \end{pmatrix}
\end{align}
\ese
All the three available GCTs can be used to bring the frame metric to a block diagonal form. There are three more KVFs:
\bse
\begin{align}
\Ze_{1} & =  \frac{z}{h^{1/3}}\partial_{x}+\Big(\ \frac{y^{2}h_{11}h_{12}+2yzh_{11}h_{22}+z^{2}h_{12}h_{22}}{2 h}+\frac{e^{2x}h_{12}}{2 h^{2/3}} \Big)\partial_{y}\notag \\
& +\Big(\ \frac{-y^{2}h_{11}^{2}-2yzh_{11}h_{12}+z^{2}(h_{11}h_{22}-2h_{12}^{2})}{2 h}-\frac{e^{2x}h_{11}}{2 h^{2/3}} \Big)\partial_{z} \\
\Ze_{2} & =  \frac{y}{h^{1/3}}\partial_{x}+\Big(\ \frac{-2y^{2}h_{12}^{2}+y^{2}h_{11}h_{22}-2yzh_{12}h_{22}-z^{2}h_{22}^{2}}{2 h}-\frac{e^{2x}h_{22}}{2 h^{2/3}} \Big)\partial_{y}\notag\\
& +\Big(\ \frac{y^{2}h_{11}h_{12}+2yzh_{11}h_{22}+z^{2}h_{12}h_{22}}{2 h}+\frac{e^{2x}h_{12}}{2 h^{2/3}} \Big)\partial_{z} \\
\Ze_{3} & =  -\Big(\ \frac{y h_{12}+z h_{22}}{h^{2/3}} \Big)\partial_{y}+\Big(\ \frac{y h_{11}+z h_{12}}{h^{2/3}} \Big)\partial_{z}  \\
h & = (h_{11}h_{22}-h_{12}^{2})^{3/2}
\end{align}
\ese
Therefore, Bianchi Type V admits a $G_{6}$ local Lie group of transformations, with the corresponding algebra (only non vanishing commutators are given):
\be\label{Type_V_G6}
\begin{array}{ccc}
&[\Ksi_{2},\Ksi_{3}] =\Ksi_{2} & [\Ksi_{1},\Ksi_{3}]=\Ksi_{1} \\
& & \\
&[\Ze_{2},\Ze_{3}]=\frac{h_{22}}{h^{2/3}} \Ze_{1}+\frac{h_{12}}{h^{2/3}} \Ze_{2} & [\Ze_{3},\Ze_{1}] =\frac{h_{12}}{h^{2/3}} \Ze_{1}+\frac{h_{11}}{h^{2/3}} \Ze_{2}\\
& & \\
&[\Ksi_{1},\Ze_{1}]=-\frac{h_{11}}{h^{1/3}}\Ze_{3} & [\Ksi_{2},\Ze_{1}]=\frac{1}{h^{1/3}}\Ksi_{3}-\frac{h_{12}}{h^{1/3}}\Ze_{3} \\
& & \\
&  [\Ksi_{3},\Ze_{1}] =\Ze_{1}  & [\Ksi_{1},\Ze_{2}]= \frac{h_{12}}{h^{1/3}}\Ze_{3}+\frac{1}{h^{1/3}}\Ksi_{3} \\
& & \\
&  [\Ksi_{2},\Ze_{2}]= \frac{h_{22}}{h^{1/3}}\Ze_{3} &[\Ksi_{3},\Ze_{2}]= \Ze_{2} \\
& & \\
& [\Ksi_{1},\Ze_{3}]= \frac{h_{11}}{h^{2/3}}\Ksi_{2}-\frac{h_{12}}{h^{2/3}}\Ksi_{1} &[\Ksi_{2},\Ze_{3}]= \frac{h_{12}}{h^{2/3}}\Ksi_{2}-\frac{h_{22}}{h^{2/3}}\Ksi_{1} 
\end{array}
\ee
This result is in full agreement with the fact that the space is not only of constant curvature but also conformally flat. 

\subsubsection*{Bianchi Type VI}
This type is characterised by:
\bse
\begin{align}
C^{1}_{\phantom{1}13} & =1, C^{2}_{\phantom{1}23}=q (\neq 0,1)\\
\{\Ksi_{A}\} & = \{\partial_{y}, \partial_{z}, \partial_{x}+y\partial_{y}+q z\partial_{z}\}\\
\{\X_{A}\} & =  \{e^{x}\partial_{y}, e^{q x}\partial_{z}, \partial_{x}\}\\
\{\Si^{A}\} & = \{e^{-x}dy, e^{-q x}dz, dx\}\\
h^{\text{irred.}}_{AB}|_{q\neq -1} & = \begin{pmatrix}
                                              h_{11} & 0           & 0\\
                                              0          & h_{22} & 0 \\
                                              0          & 0           & h_{33}
                                              \end{pmatrix}\\
h^{\text{irred.}}_{AB}|_{q=-1} & = \begin{pmatrix}
                                              h_{11} & h_{12}  & 0\\
                                              h_{12} & h_{11} & 0 \\
                                              0          & 0           & h_{33}
                                              \end{pmatrix}
\end{align}
\ese
All the three available GCTs can be used to bring the frame metric to a diagonal form if $q\neq -1$, and to a block diagonal form if $q=-1$. There is no further KVF, since the integrability conditions imply that $F_{AB}=0$, for every value of the group parameter within the range given (and for both sectors: $q\neq -1$, and $q=-1$).

\subsubsection*{Bianchi Type VII}
This type is characterised by:
\bse
\begin{align}
C^{2}_{\phantom{1}13} & =1, C^{1}_{\phantom{1}23}=-1, C^{2}_{\phantom{1}23}=q (q^{2}<4)\\
\{\Ksi_{A}\} & =  \{\partial_{y}, \partial_{z}, \partial_{x}-z\partial_{y}+(y+q z)\partial_{z}\}\\
\{\X_{A}\} & =  \{A_{1}\partial_{y}-B\partial_{z}, B\partial_{y}+A_{2}\partial_{z}, \partial_{x}\}\\
\{\Si^{A}\} & =  \{C_{1}dy-Ddz, Ddy+C_{2}dz, dx\}\\
&\text{where:}\\
A_{1,2} & =e^{kx}cos(ax)\pm kB, B=-\frac{1}{a}e^{kx}sin(ax),\\
C_{1,2} & =e^{-kx}cos(ax)\mp kD, D=-\frac{1}{a}e^{-kx}sin(ax),\\
 k & =\frac{q}{2} , a=\frac{1}{2}(4-q^{2})^{1/2}\\
h^{\text{irred.}}_{AB} & = \begin{pmatrix}
                                              h_{11} & 0           & 0\\
                                             0           & h_{22} & 0 \\
                                              0          & 0           & h_{33}
                                              \end{pmatrix}
\end{align}
\ese
All the three available GCTs can be used to bring the frame metric to a diagonal form. There is no further KVF, since the integrability conditions imply that $F_{AB}=0$, for every value of the group parameter within the range given (and for both sectors: $q\neq 0$, and $q=0$).

\subsubsection*{Bianchi Type VIII}
This type is characterised by:
\bse
\begin{align}
C^{1}_{\phantom{1}23} & =-C^{2}_{\phantom{1}31}=-C^{3}_{\phantom{1}12}=-1\\
\{\Ksi_{A}\} =\{& \frac{e^{-z}}{2}\partial_{x}+\frac{1}{2}(e^{z}-y^{2}e^{-z})\partial_{y}-ye^{-z}\partial_{z},\partial_{z},\notag \\
& \frac{e^{-z}}{2}\partial_{x}-\frac{1}{2}(e^{z}+y^{2}e^{-z})\partial_{y}-ye^{-z}\partial_{z}\}\\
\{X_{A}\} =\{& \frac{1}{2}(1+x^{2})\partial_{x}+\frac{1}{2}(1-2xy)\partial_{y}-x\partial_{z},-x\partial_{x}+y\partial_{y}+\partial_{z},\notag \\
& \frac{1}{2}(1-x^{2})\partial_{x}+\frac{1}{2}(-1+2xy)\partial_{y}+x\partial_{z}\}\\
\{\Si^{A}\} =\{&  dx+(1+x^{2})dy+(x-y-x^{2}y)dz,2xdy+(1-2xy)dz,\notag\\
&  dx+(-1+x^{2})dy+(x+y-x^{2}y)dz\}\\
h^{\text{irred.}}_{AB} & = \begin{pmatrix}
                                              h_{11} & 0          & 0\\
                                              0          & h_{22} & 0 \\
                                              0          & 0           & h_{33}
                                              \end{pmatrix}
\end{align}
\ese
All the three available GCTs can be used to bring the frame metric to a diagonal form; the degrees of freedom left are just the three eigenvalues of the matrix. The generators of those infinitesimal GCTs define the $\mathfrak{so}(2,1)$ Lie algebra. There is no further KVF, since the integrability conditions imply that $F_{AB}=0$.

\subsubsection*{Bianchi Type IX}
This type is characterised by:
\bse
\begin{align}
C^{1}_{\phantom{1}23} & =C^{2}_{\phantom{1}31}=C^{3}_{\phantom{1}12}=1\\
\{\Ksi_{A}\} = \{& \partial_{y}, cos(y)\partial_{x}-cot(x)sin(y)\partial_{y}+\frac{sin(y)}{sin(x)}\partial_{z},\notag\\
& -sin(y)\partial_{x}-cot(x)cos(y)\partial_{y}+\frac{cos(y)}{sin(x)}\partial_{z}\}\\
\{\X_{A}\} =\{& -sin(z)\partial_{x}+\frac{cos(z)}{sin(x)}\partial_{y}-cot(x)cos(z)\partial_{z},\notag\\
& cos(z)\partial_{x}+\frac{sin(z)}{sin(x)}\partial_{y}-cot(x)sin(z)\partial_{z},\partial_{z} \}\\
\{\Si^{A}\}=\{&-sin(z)dx+sin(x)cos(z)dy,cos(z)dx+sin(x)sin(z)dy,\notag\\
& cos(x)dy+dz\}\\
h^{\text{irred.}}_{AB} & =  \begin{pmatrix}
                                              h_{11} & 0          & 0\\
                                              0          & h_{22} & 0 \\
                                              0          & 0           & h_{33}
                                              \end{pmatrix}
\end{align}
\ese
All the three available GCTs can be used to bring the frame metric to a diagonal form; the degrees of freedom left are just the three eigenvalues of the matrix. The generators of those infinitesimal GCTs define the $\mathfrak{so}(3)$ Lie algebra. There is no further KVF, since the integrability conditions imply that $F_{AB}=0$.

\vspace{0.3cm}
At this point, it should be noted that all the results found thus far are in full agreement with those concerning the corresponding generic cases, in reference \cite{Szafron}. This coincidence on the results obtained via different approaches exhibits the correctness of the method.

\subsection{Application: special solutions to the EFEs}

As it is mentioned in the Introduction, in many instances symmetry can be considered as an exact assumption towards a simplification of the problem under consideration. Of course, the EFEs ---as a system of entangled PDEs of the 2nd order--- constitute a prominent example where this practice is applied. Thus, various prototypes characterised by some symmetry (of either higher order, like locally homogeneous space(times), or lower order, like space times with axial symmetry, etc.) emerge, offering a useful insight into some basic features of both the mathematical structure and the physical system described by the theory. 

In mathematical cosmology, spatial homogeneity is ---in most studies--- a sine qua non approximation to reality, since the adopted symmetry is not of higher order (thus the corresponding theory is not that much artificial), yet this order is enough for a simplification of the initial system.
\newpage
\noindent
For the development, the following would be useful:
\begin{defn}
Let a structure $(\mathcal{M}\times\mathbb{R},^{(4)}\mathbf{g},G_{3})$ be such that:
\begin{itemize}
\item[$SH_{1}$] $\mathcal{M}$ is a 3 dimensional, (usually) simply connected, Hausdorff and $C^{\infty}$ manifold,
\item[$SH_{2}$] $^{(4)}\mathbf{g}$ is a $C^{m}$ metric tensor field, defined on the entire product space that is a non degenerate, covariant tensor field of order 2, with the property that at each point of 
                           $\mathcal{M}\times\mathbb{R}$ one can choose a frame of 4 real vectors $\{\e_{1},\ldots,\e_{4}\}$, such that $\mathbf{g}(\e_{\alpha},\e_{\beta})=\eta_{\alpha\beta}$ where    
                           $\eta_{\alpha\beta}$ is a (possibly constant) symmetric matrix of Lorentzian signature,
\item[$SH_{3}$] $G_{3}$ is a local Lie group of transformations (associated with a Lie algebra $\mathfrak{g}_{3}$) acting simply transitively on $\mathcal{M}$.
\end{itemize}
Then this structure defines a spatially (and locally) homogeneous space time.
\end{defn}
A standard choice for the first fundamental form, associated with $^{(4)}\mathbf{g}$, is (in a local coordinate system $\{x^{\alpha}\}=\{t,x,y,z\}$):
\be
\begin{aligned}
^{(4)}g & =\eta_{\alpha\beta}\Th^{\alpha}\otimes\Th^{\beta}\\
\eta_{\alpha\beta} & =\begin{pmatrix}
                                   -1 & 0 \\
                                    0 & h^{\text{irred.}}_{AB}(t)
                                   \end{pmatrix}\\
\{\Th^{\alpha}\} & =\{N(t)dt,N^{A}(t)dt+\Si^{A}\}
\end{aligned}
\ee
where each pair $(h^{\text{irred.}}_{AB},\{\Si^{A}\})$ is to be attributed according to the Bianchi Type chosen. It should be noted that, now, there is one outer parameter entering the frame metric $\mathbf{h}^{\text{irred.}}$, which is of a time like character.\\
Regarding the KVFs, representing the action of $G_{3}$ at an infinitesimal level, a common technique is to ``promote'' them ---by construction (or definition) of the spatially homogeneous prototype-- into KVFs of the product space; this is done, at a local level, by adding a zero (temporal) component:
\be
\{\Ksi_{A}=\xi^{a}_{\phantom{1}A}(x^{b})\partial_{a}\}\rightarrow \{^{(4)}\Ksi_{A}=0\partial_{0}+\xi^{a}_{\phantom{1}A}(x^{b})\partial_{a}\}
\ee
The result of this technique is that the prolonged vector fields are KVFs of the metric tensor field $^{(4)}\mathbf{g}$, as well:
\be
\pounds_{^{(4)}\Ksi_{A}}(^{(4)}\mathbf{g})=0
\ee

Of course, this is the right point for one to ask whether such a technique could (or should) be implemented for the further KVFs $\{\Ze_{A}\}$ --if any.\\
The fact that there is a subclass of the Szekeres family in which the induced metric tensor field $\mathbf{g}$, defined on 3 dimensional space like hypersurfaces (which are conformally flat), is invariant under a $G_{6}$ while the metric tensor field $^{(4)}\mathbf{g}$ is not invariant under any isometry group $G_{n}$ (see \cite{Szekeres}) clearly shows that ---in general--- such a technique should not be implemented; at least, not when the goal is to find general solutions to EFEs with the initial symmetry setting.\\
On the other hand, prolonging structures is a very subtle and delicate matter. If one observes those cases which do admit further KVF(s) (i.e., Bianchi Types I, II, III, and V), one will see that both the KVF(s) $\{\Ze_{A}\}$, and the enlarged Lie algebra $\mathfrak{g}_{A+B}$ (i.e., the algebra spanned by the set $\{\Ksi_{A}\}\bigcup\{\Ze_{B}\}$) depend on the components of the irreducible frame metric $\mathbf{h}^{\text{irred.}}$; this is tantamount to the fact that they all be time dependent. Time dependence is an admissible feature when KVFs are concerned, but not when it refers to the (enlarged) Lie algebra. Thus, the only reasonable and minimum requirement is that the structure constants of the enlarged Lie algebra must be valid in the 4 dimensional product  space. The corresponding commutators can be thought of as being systems of PDEs with initial conditions described by the eigenvalues of the irreducible frame metric $\mathbf{h}^{\text{irred.}}$ at a given instant of time. This constraint will have much impact on the form of the metric tensor field $^{(4)}\mathbf{g}$, through the extended symmetry condition:
\be\label{Extended_Symmetry}
\pounds_{^{(4)}\Ze_{A}}(^{(4)}\mathbf{g})=0
\ee
for the aforementioned initial conditions must be preserved on time. Since the irreducible frame metric $h^{\text{irred.}}_{AB}$ (at a point) is a positive definite matrix, it will have three strictly positive eigenvalues (or eigenfunctions at a given instant of time), say $\{\lambda_{1},\lambda_{2},\lambda_{3}\}$. Yet, this is not enough information because in many types the automorphism group is not adequate to diagonalise the matrices under consideration. Thus, in most cases the irreducible frame metric $h^{\text{irred.}}_{AB}$ (at a point) will be fully equivalent to the set  $\{\lambda_{1},\lambda_{2},\lambda_{3}\}\bigcup\{\text{off diagonal components}\}$ (the 2nd set, estimated at a given instant of time), where $\lambda_{i} > 0$.\\
Naturally, two (or more) branches emerge:
\begin{itemize}
\item[$B_{1}$] The first is characterised by the fact that the set $\{\lambda_{1},\lambda_{2},\lambda_{3}\}$ is sufficient; then, without any loss of generality, one can set: 
                         $\lambda_{1}=\lambda_{2}=\lambda_{3}=1$ (i.e., $h^{\text{irred.}}_{AB}=\delta_{AB}$)

\item[$B_{2}$] The other(s) is(are) characterised by the fact that the set $\{\lambda_{1},\lambda_{2},\lambda_{3}\}$ is not sufficient, but rather arbitrary constant values must be assigned to the off diagonal 
                         terms --along with some convenient and positive values for the set $\{\lambda_{1},\lambda_{2},\lambda_{3}\}$.
\end{itemize}
So, in order to prolong the extra KVF(s) a series of steps is needed:
\begin{itemize}
\item[$S_{1}$] The matrix $h^{\text{irred.}}_{AB}$, at a given instant of time, is diagonalised. Thus, the three strictly positive eigenvalues $\{\lambda_{i}\}$ are associated with the diagonal components 
                         $h^{\text{irred.}}_{AA}$, while the off diagonal\footnote{In those Bianchi Types where the frame metric has not been brought to a diagonal form, there is only one off diagonal term.} 
                         components are parametrized by a constant --say $S$.
\item[$S_{2}$] The substitutions $\{h^{\text{irred.}}_{AA}\rightarrow \lambda_{A}, h^{\text{irred.}}_{\text{off diagonal}}\rightarrow S\}$ are applied to the extra KVFs, $\{\Ze_{A}\}$.
\item[$S_{3}$] A proper prolongation for the KVF(s) $\{\Ze_{A}\}$ is considered:
                        \bdm
                        \{\Ze_{A}=\zeta^{a}_{\phantom{1}A}(h^{\text{irred.}}_{AB}(t),x^{b})\partial_{a}\}\rightarrow \{^{(4)}\Ze_{A}=\zeta^{0}_{\phantom{1}A}(t)\partial_{t}
                       +\zeta^{a}_{\phantom{1}A}(\lambda_{i},S,x^{b})\partial_{a}\}
                        \edm
                       and various branches must be discriminated along the lines above.
\end{itemize}
Finally, the demand on constancy upon the structure coefficients determined by all the commutators of the enlarged Lie algebra, together with \eqref{Extended_Symmetry} will result in constraints not only upon the quantities $\zeta^{0}_{\phantom{1}A}(t)$ but also on the metric tensor field $^{(4)}\mathbf{g}$ as well. Of course, even in the case where the extra symmetry can be promoted to the product space, this by no means implies that the resulted metric tensor fields $^{(4)}\mathbf{g}$ will be consistent with the EFEs.

With all these in mind, one can explore the possibility of adopting the full symmetry group admitted by $\mathcal{M}$, as a symmetry for the product space endowed with a metric tensor field satisfying the EFEs for the vacuum (i.e., Ricci flat) case. Indeed, since ---as stated earlier--- out of the nine Bianchi Types only I, II, III and V admit further KVFs, only these four cases have to be considered. Leaving the tedious but straightforward calculations aside, it is found:
\begin{itemize}
\item[] \textbf{Bianchi Type I}:    Adoption of $\mathfrak{g}_{6}$ corresponding to \eqref{Type_I_G6} as a symmetry on the product space leads to one branch only, that of  
                                                   $\lambda_{1}=\lambda_{2}=\lambda_{3}=1$. Then, according to the steps described above, the extra KVFs result in trivial subclasses like conformally flat or Minkowski 
                                                   space time --as special members of the Kasner family of solutions \cite{Stephani_et_al.}.
\item[] \textbf{Bianchi Type II}:   Adoption of $\mathfrak{g}_{4}$ corresponding to \eqref{Type_II_G4} as a symmetry on the product space leads to two branches: one for  $S=0$ and one for $S\neq0$. In 
                                                   any case, the result is trivial  leading to either a special member of the Taub family of solutions \cite{Stephani_et_al.} or space times with symmetry higher than the Taub family 
                                                   --but incompatible with the EFEs.
\item[] \textbf{Bianchi Type III}:  Adoption of $\mathfrak{g}_{4}$ corresponding to \eqref{Type_III_G4} as a symmetry on the product space leads to two branches: one for  $S=0$ and one for $S\neq0$. 
                                                   Demanding the ensuing space times to be solutions to the EFEs, the first branch results in special members of the Ellis - MacCallum family, while the second leads to special 
                                                   members of the Kinnersley family. It is very interesting that both special families constitute 2 disjoint classes of solutions with Type III symmetry --see \cite{Daskalos}.  
\item[] \textbf{Bianchi Type V}:   Adoption of $\mathfrak{g}_{6}$ corresponding to \eqref{Type_V_G6} as a symmetry on the product space leads to two branches:  one for  $S=0$ and one for $S\neq0$. 
                                                   In any case, the result is trivial  leading to either a special member of the Joseph family of solutions \cite{Stephani_et_al.} or space times with symmetry higher than the Joseph 
                                                   family --but incompatible with the EFEs.
\end{itemize}


\section{Discussion}

A quite old discovery in the area of special solutions to the EFEs was the spark for inspiration towards a simple question: that of whether the imposition of a symmetry, as a working hypothesis, can induce further symmetry. The framework chosen is not abstract; it is that of differential geometry of locally homogeneous spaces, since these are of interest in many areas of physical theories with a geometrical flavour. Indeed, almost any such physical theory implements symmetry, in many instances, as an exact assumption; analytical cosmological prototypes in gravitation constitute a prominent example.

Both the initial question, adapted to this framework, and the analysis toward its answer, turn out to be quite fruitful. In the present work, the initial problem is divided into two parts: the first part concerns the need for an irreducible form for the system upon which the assumption of symmetry is implemented, while the second part concerns the search for induced symmetries admitted by that (irreducible) system. In the literature (see e.g., \cite{Szafron}) these induced symmetries, when they exist, are called \emph{internal} and have played an important r\^{o}le in the various classification schemes for space times in General Relativity per se. The first part of the problem, i.e., the need for an irreducible form, has led to a Proposition which attributes a geometric nature to the gauge degrees of freedom. From this point of view, this work is closely related to the spirit of \cite{Gauge}. On the other hand, and in contrast to previous work presented in \cite{Szafron,Milnor}, the analysis is quite general and applicable to any number of dimensions. Moreover, the gauge degrees of freedom have been detached, in a way, from the context of classical dynamics and have been given a purely geometrical character.\\
The second part of the problem exhibits the merits and the power of a frame approach to symmetries. Indeed, prompted by \cite{Chinea}, the use of \'E. Cartan's moving frames even when dealing with symmetries rendered the search for these clearer: the initial and induced symmetry are separated in both qualitative and quantitative terms (cf.\ comments on the matrix $F_{AB}$ in the symmetry system, Section 3). Thus, combining the two parts, the initial problem has been solved. 

At the level of applications, only a sample regarding spatially homogeneous space times has been given, since the simplest cases (i.e., when the number of dimensions is 3) either have already been attacked in the literature or have led to trivial cases. Yet, this is done in order not only to exhibit in a manifest way the consistency and the truth of the analysis by fully recovering Szafron's results, but also to provide a
pedantic example --as the 3-dimensional Bianchi types constitute, due to the simplicity of the calculations. Indeed, for the $n=4$ case there are 30 real Lie algebras or ``Bianchi Types'' and such a presentation not only would be very tedious and complicated to follow, but also it would not offer more, compared to the $n=3$ case, regarding the comprehension. Besides, the paradigm of the Bianchi Type III justifies the cause: in a concise and uniform way, two very special and completely disjoint classes of results naturally emerged. So, even if the particular example is old enough, this special feature revitalises the interest in the method. 

Worth mentioning is the fact that in those Bianchi Types where extra KVF(s) exist, two very special features manifest themselves:
\begin{itemize}
\item The totality of the KVFs, leading to non trivial cases upon embedding into a 4-dimensional Lorentzian space, forms a $\mathfrak{g}_{4}$ Lie algebra (corresponding to a $G_{4}$ group), implying 
        \emph{locally rotationally symmetric space times} (LRS) (see, e.g., \cite{MacCallumLRS}). By no means, this fact should be interpreted as if LRS were exhausting the extra symmetry, when existing, for 
         this phenomenon is an accidental feature taking place only because the number of (spatial) dimensions is 3. Therefore, when  $n> 3$ richer groups may be found, which obviously will include (generalised) 
         LRS as special sub cases.
\item The extra KVFs present a smooth behaviour --something which makes feasible the hope of extending the present analysis in order to include global/topological considerations.
\end{itemize}
An extended analysis based in this very last feature, as well as applications to higher dimensions or in cases where the degree of symmetry is lower, might be the objective of a future work. 


\newpage
\begin{ack}
The authors are much indebted to Associate Professor T. Christodoulakis not only for introducing them to the key observation (i.e., the existence of an extra symmetry admitted by the Bianchi Type III prototype --see \cite{Daskalos}) and thus initiating the entire problematic, but also for enlightening discussions during the preparation of this work.
\end{ack}

\begin{note}
All the results presented in the paper involve extensive (especially, for some Bianchi prototypes) calculations. These results have been checked using the Mathematica computer algebra system, along with some packages on Riemannian Geometry and Exterior Differential Systems ---written by Dr.\ S. Bonanos--- which are available at: http://www.inp.demokritos.gr/$\sim$sbonano/.
\end{note}



\begin{thebibliography}{77}

\bibitem{Daskalos} For example, see:\\
                              T. Christodoulakis \& Petros A. Terzis, Class.\ Quantum Grav.\ \textbf{24} (2007) 875-887\\
                              and the references therein for further details. 

\bibitem{Kinnersley} W. Kinnersley, J. Math.\ Phys.\ \textbf{10}(7) (1969) 1195-1203

\bibitem{Szafron} D.A. Szafron, J. Math.\ Phys.\, \textbf{22}(3) (1981) 543-548

\bibitem{Milnor} J. Milnor, Adv.\ Math.\, \textbf{21} (1976) 293-329                           

\bibitem{Stephani_et_al.} H. Stephani, D. Kramer, M. MacCallum, C. Hoenselaers and E. Hertl,\\
                                        \textit{Exact Solutions of Einstein's Field Equations},\\ 2nd ed.\ Cambridge University Press, Cambridge, 2003
                           
\bibitem{Gauge}  For extended discussions on gauge freedom in the case of (all) the 3-dimensional homogeneous space(time)s, see, e.g.,:\\
                           O. Coussaert \& M. Henneaux, Class.\ Quantum Grav.\, \textbf{10} (1993) 1607-1617;\\
                           J. M. Pons \& L. C. Shepley, Phys.\ Rev.\ D, \textbf{58} (1998) 024001-17;\\
                           T. Christodoulakis, G. Kofinas, E. Korfiatis, G. O. Papadopoulos, \& A. Paschos, J. Math.\ Phys.\, \textbf{42}(8) (2001) 3580-3608;\\
                           T. Christodoulakis, E. Korfiatis \& G. O. Papadopoulos,  Commun.\ Math.\ Phys.\, \textbf{226} (2002) 377-391\\
                           and the references therein.

\bibitem{Global}  For accounts on various global/topological considerations in the case of (all) the 3-dimensional homogeneous space(time)s, see, e.g.,:\\
                            G. F. R. Ellis, Gen.\ Rel.\ Grav.\, \textbf{2}(1) (1971) 7-21;\\
                            W. P. Thurston, Bull.\ Amer.\ Math.\ Soc.\, \textbf{6}(3) (1982) 357-381\\
                            and the references therein.

\bibitem{GlobalandGauge}  For extended discussions on both global/topological considerations and gauge freedom in the case of (all) the 3-dim homogeneous space(time)s, see, e.g.,:\\
                                           A. Ashtekar \& J. Samuel, Class.\ Quantum Grav.\, \textbf{8} (1991) 2191-2215;\\
                                           H. Kodama, Prog.\ Theor.\ Phys.\, \textbf{107}(2) (2002) 305-362\\
                                           and the references therein.

\bibitem{ContinuousGroupsOfTransoformations} An extended and clear treatise, without global/topological considerations, is:\\ 
                                                                           L.P. Eisenhart, \textit{Continuous Groups of Transformations},\\ Dover Publications, New York, 2003, unaltered reprint of the 1961 Dover unabridged 
                                                                           and  corrected republication of the work first published by Princeton Univesity Press in 1931\\
                                                                           While, classical references with global/topological considerations are:\\
                                                                           R. Gilmore, \textit{Lie Groups, Lie Algebras, and Some of Their Applications}, Dover Publications, New York, 2006;\\
                                                                           C. Chevalley, \textit{Theory of Lie Groups}, Vol. I, Princeton, N. J.: Princeton Univ. Press, 1946;\\
                                                                           C. Chevalley, \textit{Th\'{e}orie des Groupes de Lie}, Vol. II. \textit{Groupes Alg\'{e}briques}, Paris: Hermann, 1951;\\
                                                                           C. Chevalley, \textit{Th\'{e}orie des Groupes de Lie}, Vol. III. \textit{Th\'{e}or\`{e}mes G\'{e}n\'{e}raux sur les Alg\`{e}bres de Lie}, Paris: Hermann,  
                                                                           1955;\\
                                                                           L. S. Pontryagin, \textit{Topological Groups}, Translated from the second Russian edition by Arlen Brown, Gordon and Breach, Inc., New York, 1966;\\
                                                                           S. Lie (unter Mitwirkung von F. Engel), \textit{Theorie der Transformationsgruppen}, B\"ande I(1888), II(1890), III(1893), Leipzig: Teubner;\\
                                                                           F.W. Warner, \textit{Foundations of Differentiable Manifolds and Lie Groups}, Corrected reprint of the 1971 edition, Springer, New York, 1983;\\
                                                                           S. Helgason, \textit{Differential Geometry, Lie groups, and Symmetric Spaces}, Corrected reprint of the 1978 original, Amer. Math. Soc.,                    
                                                                           Providence, RI, 2001

\bibitem{Cartan} \'{E}. Cartan, \textit{Le\c cons sur la G\'{e}om\'{e}trie des Espaces de Riemann}, 2nd ed.\, Gauthier-Villars, Paris, 1951;\\
                           or: \textit{Geometry of Riemannian Spaces}, (english translation by J. Glazebrook, note and appendices by R. Hermann), Math.\ Sci.\ Press, Brookline MA, 1983

\bibitem{Chinea} F.J. Chinea, Class.\ Quantum Grav.\ \textbf{5} (1988) 135-145

\bibitem{Homogeneity} W. Ambrose \& I.M. Singer, Duke Math.\ J.\, \textbf{25}(4) (1958) 647-669;\\
                                     I.M. Singer, Comm.\ Pure Appl.\ Math.\, \textbf{13} (1960) 685-697

\bibitem{RyanShepley} M.P. Ryan \& L.C. Shepley, \textit{Homogeneous Relativistic Cosmologies},\\ Princeton University Press, Princeton NJ, 1975

\bibitem{Szekeres} P. Szekeres, Commun.\ Math.\ Phys.\, \textbf{41} (1975) 55-64;\\
                              W.B. Bonnor, A.H. Sulaiman, \& N. Tomimura, Gen.\ Rel.\ Grav.\, \textbf{8}(8) (1977) 549-559;\\
                              D.A. Szafron, C.B. Collins, J. Math.\ Phys.\, \textbf{20}(11) (1979) 2354-2361

\bibitem{MacCallumLRS} M.A.H. MacCallum, \textit{Anisotropic and inhomogeneous relativistic cosmologies} (pp. 533--580)\\
                                         in: \textit{General Relativity, An Einstein centenary survey}, ed.\ S.W. Hawking \& W. Israel, Cambridge University Press, Cambridge, 1979;\\
                                         M.A.H. MacCallum, \textit{The mathematics of anisotropic spatially-homogeneous cosmologies} (pp. 1--59)\\
                                         in: \textit{Physics of the Expanding Universe}, ed.\ M. Demia\'{n}ski, Springer-Verlag, (Lecture Notes in Physics No.\ 109), Berlin, 1979

\end{thebibliography}
\end{document}